\documentclass[11pt]{amsart}

\usepackage[a4paper]{geometry}
\usepackage{fancyhdr}
\usepackage{latexsym, amsthm, amsfonts, amsmath,amssymb,graphicx,lmodern,mathrsfs}

\usepackage[latin1]{inputenc} 
\usepackage[T1]{fontenc} 
\usepackage[english]{babel} 

\newtheorem{theo}{Theorem}[section]
\newtheorem{lem}[theo]{Lemma}
\newtheorem{propo}[theo]{Proposition}

\theoremstyle{definition}

\theoremstyle{remark}
\newtheorem{rem}[theo]{Remark}
\newtheorem{ex}[theo]{Example}

\makeatletter

\@addtoreset{equation}{section}
\makeatother

\def\Z{\mathbb{Z}}
\def\C{\mathbb{C}}
\def\N{\mathbb{N}}
\def\Q{\mathbb{Q}}

\def\n{\eta}

\def\s{\sigma}
\def\a{\alpha}
\def\e{\varepsilon}

\def\f{\varphi}

\def\b{\beta}

\def\n'{\nu}

\def\g{\gamma}

\def\D{\Delta}

\def\P{\Phi}

\def\px {\partial_{z}}
\def\pt {\partial_{t}}
\def\dt {\Delta_{t}}

\begin{document}
\sloppy
\title{Computing the Galois group of some parameterized linear differential equation of order two}

\author{Thomas Dreyfus}
\address{Université Paris Diderot - Institut de Mathématiques de Jussieu,}
\curraddr{4, place Jussieu 75005 Paris.}
\email{tdreyfus@math.jussieu.fr.}
\thanks{Work partially supported by ANR-06-JCJC-0028 and NFS CCF-0952591}


\date{September 13, 2013.}


\begin{abstract}
 We extend Kovacic's algorithm to compute the differential Galois group of some second order parameterized linear differential equation. In the case where no Liouvillian solutions could be found, we give a necessary and sufficient condition for the integrability of the system. We give various examples of computation.
\end{abstract} 

\maketitle
\setcounter{tocdepth}{1}
\tableofcontents
\pagebreak[3]
\section*{Introduction}

Let us consider the linear differential equation
$$
 \begin{pmatrix}
\px Y(z)   \\ 
\px^{2} Y(z)
\end{pmatrix}=
\begin{pmatrix}
0 & 1 \\ 
r(z)&0
                              \end{pmatrix}\begin{pmatrix}
Y(z)   \\ 
\px Y(z)
\end{pmatrix},
$$
where $r(z)$ is a rational function with coefficients in $\C$. We have a Galois theory for this type of equation; see \cite{VdPS}. In particular, we can associate to this equation a group $H$, which we call the differential Galois group, that measures the algebraic relations of the solutions. In this case, this group can be viewed as a linear algebraic subgroup of $\rm SL_{2}(\C)$.   
 Kovacic in \cite{Kov} (see also \cite{VdP}) uses the classification  of the linear algebraic subgroup of $\rm SL_{2}(\C)$ to obtain an algorithm that determines the Liouvillian solutions, which are the solutions that involve exponentials, indefinite integrals and solutions of polynomial equations.
 In particular, four cases happen:

\begin{enumerate}
\item $H$ is conjugated to a subgroup of $B=\left\{ \begin{pmatrix}
a&b \\ 
0&a^{-1}
\end{pmatrix}\hbox{, where } a\in \C^{*},b\in \C \right\}$ and there exists a Liouvillian solution of the form  $e^{\int_{0}^{z} f(u) du}$, with $f(z)\in \C(z)$.
\item  $H$ is conjugated to a subgroup of $$D_{\infty}=\left\{ \begin{pmatrix}
a &0 \\ 
0&a^{-1}
\end{pmatrix} \bigcup \begin{pmatrix}
0 &b^{-1} \\ 
-b&0
\end{pmatrix}\hbox{, where } a,b\in \C^{*} \right\}$$ and there exists a Liouvillian solution of the form  $e^{\int_{0}^{z} f(u) du}$, where $f(z)$ is algebraic over  $\C(z)$ of degree two  and $f(z) \notin \C(z)$.
\item  $H$ is finite and all the solutions are algebraic over $\C(z)$.
\item $H=\rm SL_{2}(\C)$ and there are no Liouvillian solutions.
\end{enumerate}
Various improvements of this algorithm have been made. See for example \cite{DLR,HVdP,UW,Z}. The case where $H$ is finite has been totally solved in \cite{SU1,SU2}; see also \cite{HW}.
\par Let $\{ \partial_{0},\partial_{1},\dots,\partial_{n}\} $ be a set of $n+1$ commuting derivations. In this article, we are interested in the parameterized linear differential equation of the form
$$
 \begin{pmatrix}
\partial_{0} Y   \\ 
\partial_{0}^{2} Y
\end{pmatrix}=
\begin{pmatrix}
0 & 1 \\ 
r&0
                              \end{pmatrix}\begin{pmatrix}
Y  \\ 
\partial_{0} Y
\end{pmatrix},$$
where $r$ belongs in a suitable $ (\partial_{0},\partial_{1},\dots,\partial_{n})$-differential field. The derivations $\partial_{1},\dots,\partial_{n}$ should be thought of as derivations with respect to the parameters. We will denote by  $C$ its field of the $\partial_{0}$-constants.  In \cite{L} and \cite{CS,HS}, the authors develop a Galois theory for the parameterized linear differential equations. They define a  parameterized differential Galois group that measures the ($\partial_{1},\dots,\partial_{n}$)-differential and algebraic relations between the solutions; see $\S \ref{2sec2}$. This group can be seen as a differential group in the sense of Kolchin: this is a group of matrices whose entries lie in the differential field $C$ and satisfy a set of polynomial differential equations with coefficients in $C$. In the case of the equation $\partial_{0}^{2} Y=rY$, the Galois group will be a linear differential algebraic subgroup of $\rm SL_{2}(C)$. The goal of this paper is to extend the algorithm from Kovacic and compute the parameterized differential Galois group of the equation $\partial_{0}^{2} Y=rY$.\\\par
The article is presented as follows. In the first section, we recall some basic facts about parameterized differential Galois theory. This theory needs to use a field of the $\partial_{0}$-constants which is $(\partial_{1},\dots,\partial_{n})$-differentially closed (see \cite{CS}, Definition 3.2). We will make a stronger assumption on the field of the $\partial_{0}$-constants $C$: we will assume that $C$ is a universal $(\partial_{1},\dots,\partial_{n})$-field (see $\S \ref{2sec2}$). We use this assumption on $C$ because a field  $(\partial_{1},\dots,\partial_{n})$-differentially closed is an abstract field which has no interpretation as a field of functions. We will see in $\S 2$ that a result of Seidenberg will allow us to identify the elements of the universal $(\partial_{1},\dots,\partial_{n})$-field $C$ which we will consider as meromorphic functions on a polydisc $D$ of $\C^{n}$.
\par In the second section, we recall the result of Seidenberg which implies that the parameterized differential Galois group can be seen as a linear differential algebraic subgroup defined over a field of meromorphic functions on a polydisc $D$ of $\C^{n}$.  Since the original algorithm from \cite{Kov} can be applied if we consider rational functions having coefficients in an algebraically closed field, we apply Kovacic's algorithm for the field of rational functions having coefficients in $C$. We obtain Liouvillian solutions that can be interpreted as meromorphic functions. Then we explain how to compute the Galois group in the four cases of Kovacic's algorithm. In the case number $4$, the Galois group is Zariski dense in $\rm SL_{2}$. We recall the definition of integrable systems and the link with integrable systems and equations with a Galois group that is Zariski dense in $\rm SL_{2}$. We decrease the number of integrability conditions by showing that this is enough to check the integrability condition for the pairs of derivations $(\px,\partial)$, where $\partial$ belongs in the vectorial space spanned by the derivations with respect to the parameters. Then, we obtain an effective way to compute the Galois group in the case number $4$, see Proposition~\ref{2propo2}. We summarize the results of the section in Theorem \ref{2theo}.
\par In the last section we give various examples of computation.\\
\par After the submission of this paper, the authors in \cite{GO} has generalized Proposition~\ref{2propo2} for equations with order more than two. Moreover, Carlos E Arreche has proved some other results in touch with parameterized Kovacic's algorithm. See \cite{Ar}.\\
\subsection*{Acknowledgments}
This article has been prepared during my thesis, which is supported by the region Île-de-France. I would like to thank my advisor, Lucia Di Vizio for sharing with me a part of her knowledge. I also want to thank Jacques-Arthur Weil and William Sit, for pointing me out some mistakes and inaccuracies in this paper.
Thanks to the anonymous referee for the pertinent remarks.
\pagebreak[3]
\section{Parameterized differential Galois theory}\label{2sec2}

Let $K$ be a differential field equipped with $n+1$ commuting derivations $ \partial_{0},\dots,\partial_{n}$ and let $\D=\{ \partial_{1},\dots,\partial_{n} \}$. We will assume that its field of the $\partial_{0}$-constants $C$ is a universal $(\D)$-field with characteristic $0$; that is, a $(\D)$-field such that for any $(\D)$-field $C_{0}\subset C$, $(\D)$-finitely generated over $\Q$, and any $(\D)$-finitely generated extension $C_{1}$ of $C_{0}$, there is a $(\D)$-differential $C_{0}$-isomorphism of $C_{1}$ into $C$.  See \cite{Kol73}, Chapter $3$, Section $7$, for more details. In particular, $C$ is $(\D)$-differentially closed. 
In this section, we will recall the result from \cite{CS} of Galois theory for the parameterized linear differential equation of the form
\begin{equation}\label{2eq1}
 \begin{pmatrix}
\partial_{0} Y   \\ 
\partial_{0}^{2} Y
\end{pmatrix}=
\begin{pmatrix}
0 & 1 \\ 
r&0
                              \end{pmatrix}\begin{pmatrix}
Y  \\ 
\partial_{0} Y
\end{pmatrix},
\end{equation}
with $r\in K$.  
A parameterized Picard-Vessiot extension of the equation (\ref{2eq1}) on $K$  is a ($\partial_{0},\dots,\partial_{n}$)-differential field extension $\mathcal{K}|K$ generated over $K$ by the entries of an invertible solution matrix (we will call it a fundamental solution) and such that the field of the $\partial_{0}$-constants of $\mathcal{K}$ is equal to $C$. 
We can apply \cite{CS}, Theorem 9.5, for the equation (\ref{2eq1}), and deduce the existence and the uniqueness up to $(\partial_{0},\dots,\partial_{n})$-differential isomorphism of the parameterized Picard-Vessiot extension $\mathcal{K}|K$. If $\D=\varnothing$, we recover the usual unparameterized Picard-Vessiot extension.
\par
The parameterized (resp. unparameterized) differential Galois group $G$ (resp. $H$) is the group of field automorphisms of the parameterized Picard-Vessiot extension (resp. the unparameterized Picard-Vessiot extension) of the equation (\ref{2eq1}),  which induces the identity on $K$ and commutes with all the derivations (resp. with the derivation $\partial_{0}$). Let $U$ be a fundamental solution.
In the  unparameterized case, 
$$\{ U^{-1}\f(U), \f \in H \} $$
 is a linear algebraic subgroup of $\rm GL_{2}(C)$. In the parameterized case we find that
$$\{ U^{-1}\f(U), \f \in G \} $$
is a linear differential algebraic subgroup, that is, a subgroup of $\rm GL_{2}(C)$ which is the zero of a set of $(\D)$-differential polynomials in $4$ variables. See \cite{CS}, Theorem 9.10, for a proof. Any other fundamental solution yields another differential algebraic subgroup of $\rm GL_{2}(C)$ which are all conjugated over $\rm GL_{2}(C)$. We will identify $G$ (resp. $H$) with a linear differential algebraic subgroup of $\rm GL_{2}(C)$ (resp. with a linear algebraic subgroup of $\rm GL_{2}(C)$) for a chosen fundamental solution. The next lemma is a classical result.
\pagebreak[3]
\begin{lem}[\cite{Kov}, Section 1.3]\label{2lem2}
$G \subset \rm SL_{2}(C)$.
\end{lem}

\pagebreak
\section{Computation of the parameterized differential Galois group}\label{2sec3}

Until the end of the paper, $C$ denotes a universal $(\D)$-field equipped with $n$ commuting derivations. Let $C(z)$ be the $(\partial_{z},\D)$-differential field of rational functions in the indeterminate $z$, with coefficients in $C$, where $z$ is a $(\D)$-constant with $\px z =1$, $C$ is the field of the $\px$-constants and such that $\px$ commutes with all the derivations. 
Let us consider the parameterized linear differential equation
\begin{equation}\label{2eq2}
 \begin{pmatrix}
\px Y(z)   \\ 
\px^{2} Y(z)
\end{pmatrix}=
\begin{pmatrix}
0 & 1 \\ 
r(z)&0
                              \end{pmatrix}\begin{pmatrix}
Y(z)   \\ 
\px Y(z)
\end{pmatrix},
\end{equation}
with $r(z)\in C(z)$. We want to apply Kovacic's algorithm for the parameterized linear differential equation (\ref{2eq2}). 
Let $G\subset \rm SL_{2}(C)$ be the parameterized  differential Galois group.
The algorithm from \cite{Kov} can be applied if the field of the $\px$-constants is algebraically closed, which is the case here.  The problem is that $C$ is an abstract field which is not very convenient for the computations. In fact we have an interpretation of the elements of $C$ as meromorphic functions.
 Let $C_{1}$ be the $(\D)$-differential field generated over $\Q$ by the $z$-coefficients of $r(z)$. Using the following result of Seidenberg  (see \cite{Sei58,Sei69}) with $K_{0}=\Q$ and $K_{1}=C_{1}$, we find the existence of a polydisc $D$ of $\C^{n}$ such that the $z$-coefficients of $r(z)$ can be considered as meromorphic functions on $D$. 
\pagebreak[3]
\begin{theo}[Seidenberg]
Let $\Q\subset K_{0} \subset K_{1}$ be finitely generated $(\D)$-differential extensions of $\Q$ and assume that $K_{0}$ consists of meromorphic functions on some domain $\Omega$ of $\C^{n}$. Then $K_{1}$ is isomorphic to the field $K_{1}^{*}$ of meromorphic functions on $\Omega_{1}\subset \Omega$ such that $K_{0}|_{\Omega_{1}}\subset K_{1}^{*}$, and the derivations in $\D$ can be identified with the derivations with respect to the coordinates on $\Omega_{1}$.
\end{theo}

Let $(\mathcal{M}_{D}, \partial_{t_{1}},\dots,\partial_{t_{n}})$ denotes the $\D_{t}=\{ \partial_{t_{1}},\dots,\partial_{t_{n}}\}$-differential field of meromorphic functions on $D$, a polydisc of $\C^{n}$. Let $t=(t_{1},\dots,t_{n})$. The discussion above tell us that the $r(z)$ of the equation (\ref{2eq2}) can be identified with $r(z,t)$, an element of $\mathcal{M}_{D}(z)$\footnote{$\mathcal{M}_{D}(z)$ denotes the $(\px,\dt)$-differential field of rational functions with indeterminate $z$ and with coefficients in $\mathcal{M}_{D}$ such that $\px z=1$, $z$ is a $(\dt)$-constant and the field $\mathcal{M}_{D}$ is the field of the $\px$-constants.}, where $D$ is a polydisc of $\C^{n}$. 
We will consider the parameterized linear differential equation
$$
 \begin{pmatrix}
\px Y(z,t)   \\ 
\px^{2} Y(z,t)
\end{pmatrix}=
\begin{pmatrix}
0 & 1 \\ 
r(z,t)&0
                              \end{pmatrix}\begin{pmatrix}
Y(z,t)   \\ 
\px Y(z,t)
\end{pmatrix},$$
with $r(z,t)\in \mathcal{M}_{D}(z)$. The group $G$ is defined by a finite number of $(\D)$-differential polynomials. Again, using the result of Seidenberg with the $(\D)$-differential field generated over $\Q$ by the coefficients of the $(\D)$-differential polynomials that define $G$ and the $z$-coefficients of $r(z)$, we deduce that $G$ can be seen as a linear differential algebraic subgroup of $\rm SL_{2}(\mathcal{M}_{D})$. Again using the result of Seidenberg, we remark that after shrinking $D$, we can assume that if $G$ is conjugated over  $\rm SL_{2}(C)$ to $Q$, we can identify $Q$ and $G$ as linear differential algebraic subgroups of $\rm SL_{2}(\mathcal{M}_{D})$, and  they are conjugated over $\rm SL_{2}(\mathcal{M}_{D})$. 
Furthermore, we obtain that the Liouvillian solutions found are defined over the algebraic closure of $\mathcal{M}_{D}(z)$. We will compute $G$ as a linear differential algebraic subgroup of $\rm SL_{2}(\mathcal{M}_{D})$.
We recall that we have four cases to consider:
\begin{enumerate}
\item There exists a Liouvillian solution of the form
$ g(z,t)=e^{\int_{0}^{z} f(u,t) du}$,
 with $f(z,t)$ belongs to $\mathcal{M}_{D}(z)$.
\item There exists a Liouvillian solution of the form
$ g(z,t)=e^{\int_{0}^{z} f(u,t) du}$,
 where $f(z,t)$ is algebraic over  $\mathcal{M}_{D}(z)$ of degree two  and $f(z,t) \notin \mathcal{M}_{D}(z)$.
\item All the solutions are algebraic over $\mathcal{M}_{D}(z)$.
\item There are no Liouvillian solutions.
\end{enumerate}
These correspond to the four cases recalled in the introduction. Proposition 6.26 of \cite{HS} says that if we take the same fundamental solution, the Zariski closure of $G$ is the unparameterized differential Galois group. This means that in each case we are looking at the Zariski dense subgroups of the group given by usual Kovacic's algorithm. 
\pagebreak[3]
\subsection{}
We start with the case number $1$. There exists a Liouvillian solution of the form:
$$ g(z,t)=e^{\int_{0}^{z} f(u,t) du},$$
 with $f(z,t)\in \mathcal{M}_{D}(z)$. The action of $G$ on the solution $g(z,t)$ can be computed with the following lemma: 
\pagebreak[3]
\begin{lem}\label{2lem1}
Let $\s \in G$. 
\begin{trivlist}
\item  (1) Let $\a(t)\in \mathcal{M}_{D}$ and $p,q \in \N$, such that $GCD(p,q)=1$.  Then there exists $k\in \N$ such that ${\s((z-\a (t))^{p/q})=e^{\frac{2ik\pi}{q}}(z-\a (t))^{p/q}}$.
\item  (2) Let $\a(t),\b(t)\in \mathcal{M}_{D}$ and $\b (t) \notin \Q$. Then there exists $a\in \C$ and $c\in \C^{*}$ such that $\s((z-\a (t))^{\b (t)})=ce^{a\b (t)}(z-\a (t))^{\b (t)}$.
\item  (3) Let  $Q(z,t)\in \mathcal{M}_{D}(z)$. Then there exists $a\in \C^{*}$ such that $\s(e^{Q(z,t)})=ae^{Q(z,t)}$.
\end{trivlist}
\end{lem}

\begin{proof}
\begin{trivlist}
\item  (1) We use the fact the elements of $G$ are fields automorphisms that leave $\mathcal{M}_{D}$ invariant.
\item  (2) A computation shows that $$\partial_{t_{i}}(z-\a (t))^{\b (t)}=\left[\log(z-\a(t))\partial_{t_{i}} \b(t)-\frac{\partial_{t_{i}}\a(t)\b(t)}{z-\a(t)}\right](z-\a (t))^{\b (t)}.$$ The fact that $\s$ commutes with all the derivations implies the existence of $a\in \C$ and $f(t)\in \mathcal{M}_{D}$ such that 
$$\s \Big(\log(z-\a(t))\Big)=\log(z-\a(t))+a \hbox{ and }\s\Big((z-\a (t))^{\b (t)}\Big)=f(t)(z-\a (t))^{\b (t)}.$$ 
Since $\partial_{t_{i}}\s =\s\partial_{t_{i}}$, we obtain that
$$
\begin{array}{ll}
&\Big[\log(z-\a(t))\partial_{t_{i}} \b(t)+a\partial_{t_{i}} \b(t)-\frac{\partial_{t_{i}}\a(t)\b(t)}{z-\a(t)}\Big]f(t)\\\\
=&\partial_{t_{i}}f(t)+f(t)\Big[\log(z-\a(t))\partial_{t_{i}} \b(t)-\frac{\partial_{t_{i}}\a(t)\b(t)}{z-\a(t)}\Big].
\end{array}
 $$
Finally, $f(t)$ satisfies the parameterized linear differential equation
$$\partial_{t_{i}}\left(\frac{\partial_{t_{i}}f(t)}{f(t)a\partial_{t_{i}}\b(t)} \right)=0.$$
This means that $\frac{\partial_{t_{i}}f(t)}{f(t)a\partial_{t_{i}}\b(t)} =c\in \C^{*}$ and $\log f(t)=a\b(t)+\log(c)$. Then we deduce that $f(t)=ce^{a\b (t)}$.
\item  (3) We use the fact that $$\partial_{t_{i}}\s \left(e^{Q(z,t)}\right)=\s \left(\partial_{t_{i}}\left(e^{Q(z,t)}\right)\right)=\s \left(\partial_{t_{i}}\left(Q(z,t)\right)e^{Q(z,t)}\right)= \partial_{t_{i}}Q(z,t)\s \left(e^{Q(z,t)}\right).$$
The equation $\partial_{t_{i}}\s \left(e^{Q(z,t)}\right)=\partial_{t_{i}}Q(z,t)\s \left(e^{Q(z,t)}\right)$ admits $\s(e^{Q(z,t)})=ae^{Q(z,t)}$ with ${a\in \C^{*}}$ as a solution.
\end{trivlist}
\end{proof}

We deduce that the matrices of $G$ are upper triangular. We will denote by ${G_{m}\backsimeq \rm GL_{1}(\mathcal{M}_{D})}$ the multiplicative group. The proof of the following proposition is inspired by the proof of \cite{Sit}, Theorem 1.4. Let $p:G\rightarrow G_{m}$ that sends $\begin{pmatrix}
m(t)& a(t)  \\ 
0&m^{-1}(t)
\end{pmatrix}$ on $m(t)$. Let $M$ be the image of $p$ and $A\subset\mathcal{M}_{D}$ such that $$\left\{ \begin{pmatrix}
1& a(t) \\ 
0&1
\end{pmatrix}\hbox{, where } a(t)\in A \right\}$$ is the kernel of $p$. We have already computed $M$ with Lemma \ref{2lem1}. For $m(t)\in M$, let $\Gamma_{m(t)}$ be the set of $\g_{m(t)}\in \mathcal{M}_{D}$ such that 
$ \begin{pmatrix}
m(t)& \g_{m(t)} \\ 
0&m(t)^{-1}
\end{pmatrix}\in G$. We will identify $\C^{*}$ with the field of constants elements of $\mathcal{M}_{D}$. If $\C^{*} \not \subset M$, because of Lemma \ref{2lem1}, $g(z,t)\in \mathcal{M}_{D}(z)$, and we can compute explicitly $g(z,t)\displaystyle \int_{u=0}^{z} g(u,t)^{-2}du,$ which is another solution. We obtain explicitly a fundamental solution and we can compute $G$. The next proposition explains how to compute $G$ when $\C^{*}\subset M$.
\pagebreak[3]
\begin{propo}\label{2propo3}
Let us keep the same notations. Assume that $\C^{*}\subset M$. Then $G$ is conjugated to 
$$ \left\{ \begin{pmatrix}
m(t)& a(t) \\ 
0&m(t)^{-1}
\end{pmatrix}\hbox{, where } m(t)\in M, a(t)\in A \right\}.$$
\end{propo}
For the proof of the proposition, we will need the following lemmas.
\pagebreak[3]
\begin{lem}\label{2lem3}
Assume that $\C^{*}\subset M$. Let $m(t)\in M$ and $a(t)\in A$. Then $m(t)a(t)\in A$.
\end{lem}

\begin{proof}
With Lemma \ref{2lem1}, we obtain that for all $m(t)\in M$, there exists $b(t)\in M$ such that $b(t)^{2}=m(t)$. 
Let $m(t)\in M$, $b(t)^{2}=m(t)$, $\g_{b(t)}\in \Gamma_{b(t)}$,  and $a(t)\in A$. The computation
$$\begin{pmatrix}
b(t)&\g_{b(t)}  \\ 
0&b(t)^{-1}
\end{pmatrix}
\begin{pmatrix}
1& a(t)  \\ 
0&1
\end{pmatrix}
\begin{pmatrix}
b(t)& \g_{b(t)}  \\ 
0&b(t)^{-1}
\end{pmatrix}^{-1}=
\begin{pmatrix}
1& m(t)a(t)  \\ 
0&1
\end{pmatrix}
$$
shows that if $m(t)\in M$ and $a(t)\in A$, then $m(t)a(t)\in A$.
\end{proof}
\pagebreak[3]
\begin{lem}\label{2lem4}
Assume that $\C^{*}\subset M$. Let $m(t)\in M$. Then $\g_{m(t)},\g'_{m(t)}\in \Gamma_{m(t)}$ if and only if ${(\g_{m(t)}-\g'_{m(t)})\in A}$.
\end{lem}

\begin{proof}
Let $\g_{m(t)},\g'_{m(t)}\in \Gamma_{m(t)}$. The computation
$$\begin{pmatrix}
m(t)& \g_{m(t)}  \\ 
0&m(t)^{-1}
\end{pmatrix}\begin{pmatrix}
m(t)& \g'_{m(t)} \\ 
0&m(t)^{-1}
\end{pmatrix}^{-1}=
\begin{pmatrix}
1& m(t)(\g_{m(t)}-\g'_{m(t)})  \\ 
0&1
\end{pmatrix}
$$
shows that $m(t)(\g_{m(t)}-\g'_{m(t)})\in A$, and then $(\g_{m(t)}-\g'_{m(t)})\in A$, because of Lemma~\ref{2lem3}. Conversely, if $(\g_{m(t)}-\g'_{m(t)})\in A$ and $\g_{m(t)}\in \Gamma_{m(t)}$, then $m(t)(\g_{m(t)}-\g'_{m(t)})\in A$, because of Lemma~\ref{2lem3}. The same computation shows that $\g'_{m(t)}\in \Gamma_{m(t)}$.
\end{proof}
\pagebreak[3]
\begin{lem}\label{2lem5}
Assume that $\C^{*}\subset M$. Let $b\in \C^{*}\setminus \{\pm 1\}$ and $\g_{b}\in \Gamma_{b}$. Let $$\b(t)=b(b^{2}-1)^{-1}\g_{b}.$$ Then, $\b(t)(m(t)-m(t)^{-1})\in \Gamma_{m(t)},$ for all $m (t)\in M$.
\end{lem}
\begin{proof}
 Let $m(t)\in M$ and $\g_{m(t)} \in \Gamma_{m(t)}$. The computation
$$\begin{pmatrix}
b& \g_{b}  \\ 
0&b^{-1}
\end{pmatrix}\begin{pmatrix}
m(t)& \g_{m(t)}  \\ 
0&m(t)^{-1}
\end{pmatrix}\begin{pmatrix}
b& \g_{b}  \\ 
0&b^{-1}
\end{pmatrix}^{-1}\begin{pmatrix}
m(t)& \g_{m(t)}  \\ 
0&m(t)^{-1}
\end{pmatrix}^{-1}
 $$
$$=\begin{pmatrix}
1& (1-m(t)^{2})b\g_{b}- (1-b^{2})m(t)\g_{m(t)}\\ 
0&1
\end{pmatrix} $$
implies that $(1-m(t)^{2})b\g_{b}- (1-b^{2})m(t)\g_{m(t)}\in A$.
 Since $(1-b^{2})m(t)\in M$, Lemma~\ref{2lem3} implies that
$$(1-b^{2})^{-1}m(t)^{-1}(1-m(t)^{2})b\g_{b}-\g_{m(t)}=\b(t)(m(t)-m(t)^{-1})-\g_{m(t)}\in A. $$
Therefore $\b(t)(m(t)-m(t)^{-1})\in \Gamma_{m(t)}$, because of Lemma \ref{2lem4}.
\end{proof}

\begin{proof}[Proof of Proposition \ref{2propo3}]
With Lemmas \ref{2lem4} and \ref{2lem5}, we find that
$$G \simeq  \left\{ \begin{pmatrix}
m(t)& \b(t)(m(t)-m(t)^{-1})+a(t) \\ 
0&m(t)^{-1}
\end{pmatrix}\hbox{, where } m(t)\in M, a(t)\in A \right\}.$$
If we change the fundamental solution, i.e, if we conjugate $G$ over $\rm GL_{2}(\mathcal{M}_{D})$, we can simplify the expression of $G$. After conjugation by the element $P=\begin{pmatrix}
1&\b(t)\\ 
0& 1
\end{pmatrix} $, we obtain that
$$PGP^{-1}\simeq   \left\{ \begin{pmatrix}
m(t)& a(t) \\ 
0&m(t)^{-1}
\end{pmatrix}\hbox{, where } m(t)\in M, a(t)\in A \right\}.$$
\end{proof}
We now want to compute $G$ when $\C^{*}\subset M$. The computation of $M$ has already been done in Lemma \ref{2lem1}. We are  now interested in the computation of $A$, which is a linear differential algebraic subgroup of $(\mathcal{M}_{D},+)$. Cassidy classifies the linear differential algebraic subgroups of the additive group in \cite{C72}, Lemma 11. We define $\mathcal{M}_{D}[y_{1}\dots,y_{\nu}]_{\dt}$, as the ring of linear homogeneous differential polynomials. There exists $P_{1},\dots P_{m}\in \mathcal{M}_{D}[y]_{\dt}$ such that $$A=\{ a(t)\in \mathcal{M}_{D} |P_{1}(a(t))=\dots =P_{m}(a(t))=0\}.$$
We recall that $g(z,t)\displaystyle \int_{u=0}^{z} g(u,t)^{-2}du$ is another solution. We can choose $\b(t)\in \mathcal{M}_{D}$ such that in the basis formed by the solutions $g(z,t)$ and $g(z,t)\displaystyle \int_{u=0}^{z} g(u,t)^{-2}du+\b(t)g(z,t)$, $G$ is equal to $\left\{ \begin{pmatrix}
m(t)& a(t) \\ 
0&m(t)^{-1}
\end{pmatrix}\hbox{, where } m(t)\in M, a(t)\in A \right\}$. 
Let $G^{g}\subset G$ be the subfield of elements that fix $g(z,t)$ and let $\s\in G^{g}$. Let $a(t)\in A$ be such that $$
\begin{array}{ll}
&\s \left(g(z,t)\displaystyle \int_{u=0}^{z} g(u,t)^{-2}du+\b(t)g(z,t)\right)\\\\
=&\left(g(z,t)\displaystyle \int_{u=0}^{z} g(u,t)^{-2}du+\b(t)g(z,t)\right)+a(t)g(z,t).
\end{array}
$$
Since $$\s \left(g(z,t)\displaystyle \int_{u=0}^{z} g(u,t)^{-2}du+\b(t)g(z,t)\right)=g(z,t)\left(\s\left(\displaystyle \int_{u=0}^{z} g(u,t)^{-2}du\right)+\b(t)\right),$$ we deduce that
$$\s\left(\displaystyle \int_{u=0}^{z} g(u,t)^{-2}du\right)-\displaystyle \int_{u=0}^{z} g(u,t)^{-2}du=a(t)\in A. $$
 Therefore, the differential polynomials $P_{i}$ satisfy $\forall \s \in G^{g}$:
$$\begin{array}{rcl}
\s \left(P_{i} \left( \int_{u=0}^{z} g(u,t)^{-2}du \right)\right)&=&P_{i} \left(\s \left(\int_{u=0}^{z} g(u,t)^{-2}du\right)\right)\\
&=&P_{i} \left( \int_{u=0}^{z} g(u,t)^{-2}du +a(t)\right)\\
&=&P_{i} \left( \int_{u=0}^{z} g(u,t)^{-2}du\right).
\end{array}$$
Since $P_{i}\left(\int_{u=0}^{z} g(u,t)^{-2}du\right)$ is fixed by the elements of $G^{g}$, we deduce by the Galois correspondence in the parameterized differential Galois theory (see \cite{CS}, Theorem~9.5) that
 $$P_{i}(a(t))=0 \Longleftrightarrow P_{i}\left(\int_{u=0}^{z} g(u,t)^{-2}du\right) \in \mathcal{M}_{D}(z)\langle g(z,t) \rangle_{\px,\dt},$$ 
where $\mathcal{M}_{D}(z)\langle g(z,t) \rangle_{\px,\dt}$ denotes the $(\px,\dt)-$differential field generated by $\mathcal{M}_{D}(z)$ and $g(z,t)$.
\pagebreak[3]
\subsection{} Let us consider the case number $2$. There exists a Liouvillian solution of the form $e^{\int_{0}^{z} f(u,t) du}$, 
such that $f(z,t)$ satisfies $$f(z,t)^{2}+a(z,t)f(z,t)+b(z,t)=0,$$ where $a(z,t),b(z,t)\in \mathcal{M}_{D}(z)$. There exists $\e \in \{ \pm 1 \}$ such that $$f(z,t)=\frac{-a(z,t) +\e\sqrt{a(z,t)^{2}-4b(z,t)}}{2}.$$ By computing the action of $G$ on $e^{\int_{0}^{z} \frac{-a(u,t) +\e\sqrt{a(u,t)^{2}-4b(u,t)}}{2} du}$,
we find that $e^{\int_{0}^{z} \frac{-a(u,t) -\e\sqrt{a(u,t)^{2}-4b(u,t)}}{2} du}$ is another Liouvillian solution which is linearly independent of the first one. 
By computing the action of $G$ on the second Liouvillian solution we find the existence of $M$, a linear differential algebraic subgroup of the multiplicative group $G_{m}$ such that, in the basis formed by the two Liouvillian solutions
$$G \simeq \left\{ \begin{pmatrix}
a(t) &0 \\ 
0&a^{-1}(t)
\end{pmatrix}\bigcup \begin{pmatrix}
0 &b^{-1}(t) \\ 
-b(t)&0
\end{pmatrix} \hbox{, where } a(t),b(t) \in M \right\}.$$
We are now interested in the computation of $M$. A direct computation shows that if there exists $\s \in G$ such that $\s \left(e^{\int_{0}^{z} f(u,t) du}\right)=\a(t)e^{\int_{0}^{z} f(u,t) du}$, then for all $i\leq n$, $\a(t)$ satisfies the parameterized differential equation 
$$ \partial_{t_{i}} \a(t)+\a(t) \left(\partial_{t_{i}} \int_{0}^{z} f(u,t) du\right)=\a(t)\s \left(\partial_{t_{i}} \int_{0}^{z} f(u,t) du\right).$$
Let $\widetilde{\partial_{t_{i}}}\a(t)=\frac{\partial_{t_{i}}\a(t)}{\a (t)}$ be the logarithm derivation. In \cite{C72}, Chapter $4$, we see that there exist $P_{1},\dots,P_{k}\in \mathcal{M}_{D}[y_{1}\dots,y_{n}]_{\dt}$ such that
$$M\simeq \left\{ \a(t) \Big| P_{1}\left(\widetilde{\partial_{t_{i}}}\a(t)\right)=\dots=P_{k}\left(\widetilde{\partial_{t_{i}}}\a(t)\right)=0\right\}. $$
The polynomial $P_{j}$ satisfies, for all $\s\in G$, $P_{j}\left(\partial_{t_{i}} \int_{0}^{z} f(u,t)du\right)=\s \left(P_{j}\left(\partial_{t_{i}} \int_{0}^{z} f(u,t)du\right)\right)$ and then $$P_{j}\left(\widetilde{\partial_{t_{i}}}\a(t)\right)=0 \Longleftrightarrow P_{j}\left( \partial_{t_{i}}\int_{0}^{z} f(u,t)du\right)\in \mathcal{M}_{D}(z).$$
\pagebreak[3]
\subsection{} In the third case, $G$ is finite, because whose Zariski closure is finite. Since all finite linear differential algebraic subgroups of $\rm SL_{2}(\mathcal{M}_{D})$ are finite linear algebraic subgroups of $\rm SL_{2}(\mathcal{M}_{D})$, $G$ is equal to the unparameterized differential Galois group. This is the same problem as in the unparameterized case. See \cite{HW} for the computation of $G$. 
\pagebreak[3]
\subsection{} We now consider the case where no Liouvillian solutions are found. We have seen in the introduction that in this case, the unparameterized differential Galois group is $\rm SL_{2}(\mathcal{M}_{D})$.  
Therefore, $G$ is Zariski dense in $\rm SL_{2}(\mathcal{M}_{D})$.
\par  The classification of the Zariski dense subgroup of $\rm SL_{2}(\mathcal{M}_{D})$ has been made in \cite{C72}, Proposition 42. Let $\textbf{D}$ be the $\mathcal{M}_{D}$-vectorial space of derivations of the form
$$\left\{\displaystyle \sum_{i=0}^{n} a_{i}(t) \partial_{t_{i}} \hbox{, where } a_{i}(t) \in \mathcal{M}_{D} \right\},$$ and $\mathbb{D}$ a vectorial subspace of $\textbf{D}$. Let $\mathcal{M}_{D}^{\mathbb{D}}$ be the elements of $\mathcal{M}_{D}$ that are constant for the derivations in $\mathbb{D}$. Remark that if $\mathbb{D}=\{0\}$, then $\mathcal{M}_{D}^{\mathbb{D}}=\mathcal{M}_{D}$. The linear differential algebraic subgroups of $\rm SL_{2}(\mathcal{M}_{D})$ that are Zariski dense in $\rm SL_{2}(\mathcal{M}_{D})$ are conjugated over $\rm SL_{2}(\mathcal{M}_{D})$ to the groups of the form $\rm SL_{2}(\mathcal{M}_{D}^{\mathbb{D}})$, with ${\mathbb{D}}$ a vectorial subspace of $\textbf{D}$.
\par  Let ${\mathbb{D}}\subset \textbf{D}$ be such that $G$ is conjugated over $\rm SL_{2}(\mathcal{M}_{D})$ to $\rm SL_{2}(\mathcal{M}_{D}^{\mathbb{D}})$. We want to compute explicitly ${\mathbb{D}}$. This leads us to the notion of integrable systems.
Let $A_{0}(z,t),\dots,A_{k}(z,t)$, $m \times m$ matrices with entries in $\mathcal{M}_{D}(z)$ and $ \partial'_{t_{1}},\dots, \partial'_{t_{k}}\in \textbf{D}$. The following system   
\[
[S]: \; \left \{
\begin{array}{ccl}
     \partial_{z} Y(z,t)&=&A_{0}(z,t)Y(z,t)  \\
   \partial'_{t_{1}} Y(z,t)&=&A_{1}(z,t)Y(z,t) \\
        & \vdots  &                        \\

   \partial'_{t_{k}} Y(z,t)&=&A_{k}(z,t)Y(z,t) \\
\end{array}
\right.
\]  is  integrable if and only if, for all $0\leq i,j \leq k$,
$$\partial'_{t_{j}}A_{i}(z,t)-\partial'_{t_{i}}A_{j}(z,t)=A_{j}(z,t)A_{i}(z,t)-A_{i}(z,t)A_{j}(z,t), $$ 
where $\partial'_{t_{0}}=\px$.
We recall here \cite{CS}, Proposition 6.3, which relates the integrable system and the parameterized differential Galois group in the case where the field of the $\px$-constants is differentially closed.
\pagebreak[3]
\begin{propo}\label{2propo1}Let $\{\partial'_{t_{1}},\dots,\partial'_{t_{k}}\}$ be a commuting basis of ${\mathbb{D}}$, a vectorial subspace of $\textbf{D}$. $G$ is conjugated to $\rm SL_{2}(\mathcal{M}_{D}^{\mathbb{D}})$  over $\rm SL_{2}(\mathcal{M}_{D})$
if and only if there exist $A_{1}(z,t),\dots ,A_{k}(z,t)$, $m \times m$ matrices with entries in $\mathcal{M}_{D}(z)$\footnote{Using the result of Seidenberg, we can identify the matrices as elements of $\rm GL_{2}(\mathcal{M}_{D}(z))$ because their entries involve a finite number of elements of the fields of the $\px$-constants.}, such that the following system is integrable:
 \[
[S]: \; \left \{
\begin{array}{ccl}
     \partial_{z} Y(z,t)&=&A(z,t)Y(z,t)  \\
   \partial'_{t_{1}} Y(z,t)&=&A_{1}(z,t)Y(z,t) \\
        & \vdots  &                        \\

   \partial'_{t_{k}} Y(z,t)&=&A_{k}(z,t)Y(z,t). \\
\end{array}
\right.
\]
\end{propo}

We want to give simpler necessary and sufficient conditions for the integrability of the system in Proposition \ref{2propo1}.
First, we will write a necessary and sufficient condition for the integrability of 
 \[
 [S']: \; \left \{
\begin{array}{ccl}
     \partial_{z} Y(z,t)&=&A(z,t)Y(z,t)   \\
   \partial' Y(z,t)&=&A'(z,t)Y(z,t), \\
\end{array}
\right.
\]
where $A'(z,t)=\begin{pmatrix}
a(z,t) & b(z,t) \\ 
c(z,t) & d(z,t)
\end{pmatrix}$ is an $m \times m$ matrix with entries in $\mathcal{M}_{D}(z)$ and ${\partial' \in \textbf{D}}$. The fact that $[S']$ is integrable is equivalent to the solution in $(\mathcal{M}_{D}(z))^{4}$ of the parameterized differential system:\\ 
$\begin{array}{ll}
&  
 \left \{
\begin{array}{ccl}
     \px a(z,t)&=&c(z,t)-b(z,t)r(z,t)   \\
   \px b(z,t)&=&d(z,t)-a(z,t) \\
   \px c(z,t)&=&(a(z,t)-d(z,t))r(z,t)+\partial ' r(z,t) \\
    \px d(z,t)&=&b(z,t)r(z,t)-c(z,t) \\
\end{array}
\right. \\
& \\
\Longleftrightarrow &\left \{
\begin{array}{ccl}
     \px a(z,t)&=&-\px d(z,t)   \\
   \px^{2} b(z,t)&=&2\px d(z,t) \\
   \px c(z,t)&=&-\px b(z,t) r(z,t)+\partial ' r(z,t) \\
    \frac{\px^{2} b(z,t)}{2}&=&b(z,t)r(z,t)-c(z,t) \\
\end{array}
\right.\\
& \\
\Longleftrightarrow &\left \{
\begin{array}{ccl}
      \px a(z,t)&=&-\px d(z,t)   \\
   \px^{2} b(z,t)&=&2\px d(z,t) \\
   \px c(z,t)&=&-\px b(z,t) r(z,t)+\partial ' r(z,t) \\
 \frac{\px^{3} b(z,t)}{2}&=&2\px b(z,t)r(z,t)+b(z,t)\px r(z,t)-\partial 'r(z,t).
\end{array}
\right.\\\\
\end{array}$ \\
We can easily see that the existence of $b(z,t)\in \mathcal{M}_{D}(z)$ as a solution of 
$$\frac{\px^{3} b(z,t)}{2}=2\px b(z,t)r(z,t)+b(z,t)\px r(z,t)-\partial 'r(z,t)$$
is equivalent to the fact that the system $[S']$ is integrable.
There exists an algorithm to determine if such a system has a solution (see \cite{VdPS}, p. 100). We obtain a necessary and sufficient condition on $\partial '$ for the integrability condition of the system $[S']$. Let $\mathbb{D}$ be the maximal vectorial subspace of $\textbf{D}$ such that for all derivations  $\partial '$ in $\mathbb{D}$, there exists $A'(z,t)$, $m \times m$ matrix with entries in $\mathcal{M}_{D}(z)$ such that the following system is integrable:
 \[
 [S']: \; \left \{
\begin{array}{ccl}
     \partial_{z} Y(z,t)&=&A(z,t)Y(z,t)   \\
   \partial' Y(z,t)&=&A'(z,t)Y(z,t). \\
\end{array}
\right.
\]
 We want to prove that the parameterized differential Galois group of ${\partial_{z} Y(z,t)=A(z,t)Y(z,t)}$ is conjugated to $\rm SL_{2}(\mathcal{M}_{D}^{\mathbb{D}})$ over $\rm SL_{2}(\mathcal{M}_{D})$. Assume that this is not the case. Then by Proposition \ref{2propo1}, there exists $\mathbb{D}_{1},\mathbb{D}_{2}\subsetneq \mathbb{D}$, having at least dimension $1$, with  $\mathbb{D}_{1} \neq \mathbb{D}_{2}$ such that $G$ is conjugated to $\rm SL_{2}(\mathcal{M}_{D}^{\mathbb{D}_{1}})$ and $\rm SL_{2}(\mathcal{M}_{D}^{\mathbb{D}_{2}})$. In this case, $\rm SL_{2}(\mathcal{M}_{D}^{\mathbb{D}_{1}})$ is conjugated to $\rm SL_{2}(\mathcal{M}_{D}^{\mathbb{D}_{2}})$ over $\rm SL_{2}(\mathcal{M}_{D})$. The fact that $\mathbb{D}_{1}=\mathbb{D}_{2}$ is proved in \cite{Sit}, Theorem 1.2, Chapter 2, but we will recall the proof here. Let $\a \in \mathcal{M}_{D}^{\mathbb{D}_{1}}$ and consider the diagonal matrix $M=\begin{pmatrix}
\a & 0 \\ 
0 &\a^{-1} 
\end{pmatrix}\in \rm SL_{2}(\mathcal{M}_{D}^{\mathbb{D}_{1}})$. Since similar matrices have the same set of eigenvalues and $\mathcal{M}_{D}^{\mathbb{D}_{2}}$ is algebraically closed , we obtain that $\a(t) \in \mathcal{M}_{D}^{\mathbb{D}_{2}}$. Therefore  $\mathcal{M}_{D}^{\mathbb{D}_{1}} \subset \mathcal{M}_{D}^{\mathbb{D}_{2}}$ and, by symmetry, $\mathcal{M}_{D}^{\mathbb{D}_{1}}=\mathcal{M}_{D}^{\mathbb{D}_{2}}$. We then deduce $\mathbb{D}_{1}=\mathbb{D}_{2}= \mathbb{D}$. We have proved:
\pagebreak[3]
\begin{propo}\label{2propo2}
We have the following equivalences:
\begin{trivlist}
\item (1) $G$ is conjugated to $\rm SL_{2}(\mathcal{M}_{D}^{\mathbb{D}})$ over $\rm SL_{2}(\mathcal{M}_{D})$.
\item (2) For all $\partial '$ that belongs in a commuting basis of $\mathbb{D}$, the following parameterized differential equation has a solution in $\mathcal{M}_{D}(z)$:
$$\frac{\px^{3} b(z,t)}{2}=2\px b(z,t)r(z,t)+b(z,t)\px r(z,t)-\partial 'r(z,t).$$
\item (3) For all $\partial '\in \mathbb{D}$, the following parameterized differential equation has a solution in $\mathcal{M}_{D}(z)$:
$$\frac{\px^{3} b(z,t)}{2}=2\px b(z,t)r(z,t)+b(z,t)\px r(z,t)-\partial 'r(z,t).$$
\end{trivlist}
\end{propo}
\pagebreak[3]
\begin{rem}\label{2rem1}
In the case where $n=1$, i.e, there is only one parameter, the Zariski dense subgroups of $\rm SL_{2}(\mathcal{M}_{D})$ are (up to conjugation over $\rm SL_{2}(\mathcal{M}_{D})$) $\rm SL_{2}(\mathcal{M}_{D})$ and $\rm SL_{2}(\C)$. Then we only have to check whether 
$$ \frac{\px^{3} b(z,t)}{2}=2\px b(z,t)r(z,t)+b(z,t)\px r(z,t)-\partial_{t}r(z,t), $$
has a solution in  $\mathcal{M}_{D}(z)$.
\end{rem}

\pagebreak
\subsection{} We summarize in the next theorem the results of this section.

\begin{theo}\label{2theo} Let us consider $\px^{2} Y(z,t)=r(z,t)Y(z,t)$ with $r(z,t)\in \mathcal{M}_{D}(z)$ and let $G$ be the parameterized differential Galois group, seen as a linear differential algebraic subgroup of $\rm SL_{2}(\mathcal{M}_{D})$. There are four possibilities:
\begin{enumerate}
\item There exists a Liouvillian solution of the form  $g(z,t)=e^{\int_{0}^{z} f(u,t) du}$, with $f(z,t)$ belongs to $\mathcal{M}_{D}(z)$. There are two possibilities:
\begin{enumerate}
\item If $g(z,t)\in \mathcal{M}_{D}$, then we can compute explicitly another solution $g(z,t)\displaystyle \int_{u=0}^{z}g(u,t)^{-2}du$ which is linearly independent with $g(z,t)$. In this basis of solutions we can compute explicitly $G$.
\item In the other case, $G$ is conjugated to:
$$ \left\{ \begin{pmatrix}
m(t)& a(t) \\ 
0&m(t)^{-1}
\end{pmatrix}\hbox{, where } m(t)\in M, a(t)\in A \right\},$$
where:
\end{enumerate}
$$\begin{array}{ll}
M&=\{ g(z,t)^{-1}\s  (g(z,t)), \s\in G \},\\\\
A&=\left\{a(t)\in \mathcal{M}_{D}\Bigg| 
\begin{array}{l}
\forall P \in \mathcal{M}_{D}[y]_{\dt}, \\P \left(\int_{u=0}^{z} g(u,t)^{-2}du\right) \in \mathcal{M}_{D}(z)\langle g(z,t)\rangle_{\px,\dt} \Longleftrightarrow P(a(t))=0
\end{array}\right\}.
\end{array}$$
\item  There exists a Liouvillian solution of the form  $g(z,t)=e^{\int_{0}^{z} f(u,t) du}$, where $f(z,t)$ is algebraic over  $\mathcal{M}_{D}(z)$ of degree two  and $f(z,t) \notin \mathcal{M}_{D}(z)$. In this case, $G$ is conjugated to $$\left\{ \begin{pmatrix}
a &0 \\ 
0&a^{-1}
\end{pmatrix} \bigcup \begin{pmatrix}
0 &b^{-1} \\ 
-b&0
\end{pmatrix}\hbox{, where } a,b\in M \right\}, \hbox{ where}$$
$$M=\left\{f(t)\in \mathcal{M}_{D}\Bigg|
\begin{array}{l}
\forall P \in \mathcal{M}_{D}[y_{1}\dots,y_{n}]_{\dt},\\
 P \left(\partial_{t_{i}}\int_{u=0}^{z} f(u,t)du\right) \in \mathcal{M}_{D}(z) \Leftrightarrow P\left(\widetilde{\partial}_{t_{i}}f(t)\right)=0
\end{array} \right\}.$$

\item $G$ is finite. In this case, $G$ is equal to the unparameterized differential Galois group.
\item There are no Liouvillian solutions. In this case, there exists $\mathbb{D}$, a $\mathcal{M}_{D}$-vectorial space of derivations spanned by $\D_{t}$, such that $G$ is conjugated to $\rm SL_{2}(\mathcal{M}_{D}^{\mathbb{D}})$. Moreover, $\pt '\in \mathbb{D}$ if and only if the following parameterized differential equation has a solution in $\mathcal{M}_{D}(z)$:
$$\frac{\px^{3} b(z,t)}{2}=2\px b(z,t)r(z,t)+b(z,t)\px r(z,t)-\pt 'r(z).$$
\end{enumerate}
\end{theo}

Notice that the computation of the Liouvillian solutions and the unparameterized differential Galois group are already known. Our results compute the parameterized differential Galois group in the cases 1,2 and 4. The classification of the Zariski dense linear differential algebraic subgroup of $\rm SL_{2}(\mathcal{M}_{D})$ and the link with integrable systems were already known (see \cite{C72,CS}), but we give here an effective way to compute the Galois group in the case number 4 and we decrease the number of integrability conditions.
\pagebreak
\section{Examples}

In the following examples, we will consider equations having coefficients in $\mathcal{M}_{D}(z)$ and we will compute $G$ as a linear differential algebraic subgroup of $\rm SL_{2}(\mathcal{M}_{D})$. In the three first examples, we are in the case where no Liouvillian solutions are found. In the fourth example, we are in the case number $1$ and in the last example, we are in the case number~$2$.
\pagebreak[3]
\begin{ex}[Schrodinger equation with rational potential of odd degree]
Let us consider $r(z,t)=z^{2n+1}+\displaystyle \sum_{i=0}^{2n} t_{i}z^{i}$. There are no Liouvillian solutions. The parameterized linear differential equation
$$\frac{\px^{3} b(z,t)}{2}=2\px b(z,t)r(z,t)+b(z,t)\px r(z,t)-\displaystyle \sum_{i=0}^{2n} a_{i}(t)z^{i}$$ has a rational solution if and only if there exists $c(t)\in \mathcal{M}_{D}$ such that
 \[
 \; \left \{
\begin{array}{ccccc}
    & a_{2n}(t)&=&c(t)(2n+1)   \\
 i<2n: & a_{i}(t)&=&c(t)(i+1)t_{i+1}. \\ 
\end{array} 
\right. 
\] 
Then $$G\simeq \rm SL_{2}\left(\mathcal{M}_{D}^{\partial_{t'}}\right), \hbox{ where }\partial_{t'}=(2n+1)\partial_{t_{2n}} +\displaystyle \sum_{i=0}^{2n-1} (i+1)t_{i+1}\partial_{t_{i}} .$$
\end{ex}
\pagebreak[3]
\begin{ex}[Bessel equation]
Let $r(z,t)=\frac{4t^{2}-1}{4z^{2}}-1$. 
In \cite{Kov}, $\S 4.2$, Example 2, we see that if $t\notin \frac{1}{2}+\Z$, this parameterized linear differential equation has no Liouvillian solution. We can choose $D$ such that $\{D\cap (\frac{1}{2}+\Z )\}= \varnothing$. We obtain that $G$ is Zariski dense in $\rm SL_{2}(\mathcal{M}_{D})$. With Remark \ref{2rem1}, we have to see whether the parameterized linear differential equation
\begin{equation}\label{2eq3}
\frac{\px^{3} b(z,t)}{2}=2\px b(z,t)\left(\frac{4t^{2}-1}{4z^{2}}-1\right)+b(z,t)\frac{1-4t^{2}}{2z^{3}}-\frac{2t}{z}
\end{equation}
has a solution in $\mathcal{M}_{D}(z)$. Suppose that there exists $b(z,t)\in \mathcal{M}_{D}(z)$ satisfying such an equation. We can see directly that if $b(z,t)$ has a pole, then it is $z=0$. Assume that $b(z,t)$ has a pole of order $\nu$ at $z=0$ and let $0\neq f(t)\in \mathcal{M}_{D}$ equal the value at $(0,t)$ of $z^{\nu}b(z,t)$. Since $b(z,t)$ satisfies the equation (\ref{2eq3}), we find for all $t\in D$:
$$\frac{-f(t)\nu(\nu-1)(\nu-2)}{2}=-f(t)\nu\frac{4t^{2}-1}{2}+f(t)\frac{1-4t^{2}}{2}. $$
For all $\nu$, there is no $0\neq f(t)$ satisfying this equality and we find that $b(z,t)\in \mathcal{M}_{D}[z]$. Let $\nu$ be its degree and $f(t)$ its leading term. The equation (\ref{2eq3}) has no constant solution, and we can assume $\nu>1$. We find that for all $t\in D$,
$$0=-2\nu f(t),  $$ 
which implies that the equation (\ref{2eq3}) has no solutions  in $\mathcal{M}_{D}(z)$ and then
$$G\simeq \rm SL_{2}(\mathcal{M}_{D}).$$
\end{ex}
\pagebreak[3]
\begin{ex}[Harmonic oscillator]
Let $r(z,t)=\frac{z^{2}}{4}+t$. There are no Liouvillian solutions. With Remark \ref{2rem1}, we have to check whether the parameterized linear differential equation
$$\frac{\px^{3} b(z,t)}{2}=2\px b(z,t)\left(\frac{z^{2}}{4}+t\right)+b(z,t)\frac{z}{2}-1$$ has a solution in $\mathcal{M}_{D}(z)$. We can see directly that if $b(z,t)\in \mathcal{M}_{D}(z)$ is a solution, then it has no poles, which means that $b(z,t)\in \mathcal{M}_{D}[z]$. Let $\nu$ be its degree and $0\neq f(t)$ be its leading term. We find that $\frac{(\nu+1)f(t)}{2}=0$, which admits no solution different from $0$. Then $$G\simeq \rm SL_{2}(\mathcal{M}_{D}).$$
\end{ex}
\pagebreak[3]
 \begin{ex}
If $r(z,t)=\frac{t}{z^{2}}$, then we have two Liouvillian solutions
$$f_{1}(z,t)=\sqrt{z}z^{\frac{\sqrt{1+4t}}{2}}  \hbox{ and } f_{2}(z,t)=\sqrt{z}z^{-\frac{\sqrt{1+4t}}{2}}.$$
We can compute the parameterized differential Galois group for the fundamental solution  $$\begin{pmatrix}
f_{1}(z,t) &f_{2}(z,t) \\ 
\px f_{1}(z,t) & \px f_{2}(z,t)
\end{pmatrix}:$$
$$G \simeq \left\{ \begin{pmatrix}
\a e^{a (\sqrt{1+4t})} &0 \\ 
0&\a^{-1} e^{-a (\sqrt{1+4t})} 
\end{pmatrix}\hbox{, where } a\in \C \; , \a \in \C^{*} \right\}.$$ Viewed as a linear differential algebraic subgroup $\rm GL_{2}(\mathcal{M}_{D})$, $$G\simeq  \left\{ \begin{pmatrix}
\a (t) &0 \\ 
0&\a^{-1}(t)
\end{pmatrix}\hbox{, where } \pt\left(\frac{\sqrt{1+4t} \pt \a (t)}{\a (t)}\right) =0 \right\}.$$

\end{ex}
\pagebreak[3]
\begin{ex}
If $r(z,t)=\frac{t}{z}-\frac{3}{16z^{2}}$, then we have two Liouvillian solutions
$$f_{1}(z,t)=(z)^{1/4}e^{2(tz)^{1/2}} \hbox{ and } f_{2}(z,t)=(z)^{1/4}e^{-2(tz)^{1/2}}.$$
We can compute the parameterized differential Galois group for the fundamental solution  $$\begin{pmatrix}
f_{1}(z,t) &f_{2}(z,t) \\ 
\px f_{1}(z,t) & \px f_{2}(z,t)
\end{pmatrix}:$$
$$G\simeq \left\{ \begin{pmatrix}
a(t) &0 \\ 
0&a^{-1}(t)
\end{pmatrix}\bigcup \begin{pmatrix}
0 &b^{-1}(t) \\ 
-b(t)&0
\end{pmatrix}\hbox{, where } a(t),b(t) \in \C^{*} \right\}.$$
We can remark that we have an integrable system
\[
\; \left \{
\begin{array}{ccl}
     \px Y(z,t)&=&A(z,t)Y(z,t)   \\
   \partial_{t}Y(z,t)&=&B(z,t)Y(z,t) \\
\end{array}
\right.
\] 
with
$$ A(z,t)=\begin{pmatrix}
0 & 1 \\ \\
\frac{t}{z}-\frac{3}{16z^{2}} & 0
\end{pmatrix} \hbox{ and } B(z,t)=\begin{pmatrix}
-\frac{1}{4t} & \frac{z}{t} \\ \\
1-\frac{3}{16tz} & \frac{3}{4t} 
\end{pmatrix}.  $$
\end{ex}
\fussy
\pagebreak[3]
\nocite{A,AMW}
\bibliographystyle{alpha}
\bibliography{C:/Users/thomas.dreyfus/Dropbox/Maths/biblio}

\end{document}